\documentclass[12pt]{amsart}
\usepackage{amssymb}
\usepackage{amsmath}
\usepackage{setspace}
\usepackage{mathtools}
\usepackage{verbatim}
\usepackage{amsthm}
\usepackage{tikz-cd}
\usepackage{hyperref}
\usepackage[french, english]{babel}

\newtheorem{theorem}{Theorem}[section]
\newtheorem{lemma}[theorem]{Lemma}
\newtheorem{corollary}[theorem]{Corollary}

\newtheorem{proposition}[theorem]{Proposition}

\newcommand{\bb}[1]{\mathbb{#1}}
\newcommand{\sing}[1]{\operatorname{Sing} \, #1}

\newcommand{\exc}[1]{\operatorname{Exc} \, #1}

\newcommand{\loc}[1]{\operatorname{loc} \, #1}

\newcommand{\cal}[1]{\mathcal{#1}}

\makeatletter
\newenvironment{abstracts}{%
  \ifx\maketitle\relax
    \ClassWarning{\@classname}{Abstract should precede
      \protect\maketitle\space in AMS document classes; reported}%
  \fi
  \global\setbox\abstractbox=\vtop \bgroup
    \normalfont\Small
    \list{}{\labelwidth\z@
      \leftmargin3pc \rightmargin\leftmargin
      \listparindent\normalparindent \itemindent\z@
      \parsep\z@ \@plus\p@
      
      \itemsep\medskipamount
    }%
}{%
  \endlist\egroup
  \ifx\@setabstract\relax \@setabstracta \fi
}

\newcommand{\abstractin}[1]{%
  \otherlanguage{#1}%
  \item[\hskip\labelsep\scshape\abstractname.]%
}
\makeatother

\DeclareMathOperator{\Aut}{Aut}

\title{On base point freeness   for rank one foliations}
\author{Paolo Cascini and Calum Spicer} 

\address{Department of Mathematics, Imperial College London, 180 Queen's Gate, 
London SW7 2AZ, UK} 
\email{p.cascini@imperial.ac.uk}

\address{Department of Mathematics, King's College London, Strand,
London WC2R 2LS, UK}
\email{calum.spicer@kcl.ac.uk}

\dedicatory{To James M\textsuperscript cKernan on the occasion of his 60th birthday.}

\begin{document}

\begin{abstracts}

\abstractin{english}
We prove the base point free theorem for log canonical foliated pairs of rank one  
on a $\mathbb Q$-factorial projective klt threefold.
Moreover, we show abundance in the case of numerically trivial log canonical foliated  pairs of rank one in any dimension.  
\end{abstracts}

\maketitle

\tableofcontents

\section{Introduction}

In recent years, the Minimal Model Program (MMP) has been extended beyond its classical realm to encompass the birational classification of foliated varieties. Foliations of rank one arise naturally in both dynamics and algebraic geometry, and understanding their positivity properties is a crucial step toward a comprehensive birational theory. 

\medskip

Our first main result establishes a base point free theorem for log canonical, rank-one foliated pairs on threefolds.  This extends the classical Kawamata–Shokurov base point free theorem to the foliated setting (see also \cite[Theorem 1.3]{CS21} for the case of co-rank one foliations over a threefold and \cite[Theorem 1.3]{Li25} for a more general result for rank one foliations over a threefold):

\begin{theorem}
[=Theorem \ref{t_bpf}]
\label{t_main1}
Let $X$ be a normal projective threefold with $\mathbb Q$-factorial klt singularities 
and let $(\cal F,\Delta)$ be a rank one  foliated  pair on $X$ with log canonical singularities. Assume that  $\Delta=A+B$ where $A$ is an ample $\bb Q$-divisor and $B\ge 0$ is a $\bb Q$-divisor. 
Suppose that $K_{\cal F}+\Delta$ is nef.  

Then $K_{\cal F}+\Delta$ is semi-ample.
\end{theorem}

Our second main theorem addresses the abundance problem in arbitrary dimension when the adjoint class is numerically trivial
 (see \cite[Theorem 1.2]{gongyo13} and the reference therein for the absolute case,   \cite[Theorem 1.7]{CS21} for the case of foliations of co-rank one  over a threefold):

\begin{theorem}[= Theorem \ref{t_main}]
\label{t_main0}
Let $X$ be a normal projective $\mathbb Q$-factorial klt variety and let $(\cal F, \Delta)$
be a rank one  foliated pair on $X$ with log canonical singularities.
Suppose that $K_{\cal F}+\Delta \equiv 0$.  

Then $K_{\cal F}+\Delta \sim_{\bb Q} 0$.
\end{theorem}

Note that Theorem \ref{t_main1} and a version of Theorem \ref{t_main0} in dimension three appeared  in an earlier version of our paper \cite{CS20}. 
A different proof of Theorem \ref{t_main0} can also be found  in \cite[Theorem 5.1]{druel2024numericallyflatfoliationsholomorphic}. 

\subsection{Acknowledgements}
The first author is partially supported by a Simons collaboration grant. 
The second author is partially funded by EPSRC. We would like to thank Fabio Bernasconi,  Mengchu Li,  Jihao Liu and Jorge Pereira
for many useful discussions. We would also like to thank the referee for careful reading the paper and for
several useful comments.
%We are grateful to the referee for carefully reading the paper
%and for several useful suggestions and corrections.

\section{Preliminary Results}

\subsection{Notations}
We work over the field of complex numbers $\mathbb C$. 
We refer to \cite{KM98} for the classical definitions of singularities that appear in the minimal model program.

Given a normal variety $X$, we denote by $\Omega^1_X$ its sheaf of K\"ahler differentials and
by $T_X:=(\Omega^1_X)^*$ its tangent sheaf. 
A {\bf foliation of rank one} on a normal variety $X$ is a rank one coherent subsheaf $T_{\cal F}\subset T_X$ such that 
$T_{\cal F}$ is saturated in $T_X$.
The {\bf canonical divisor} of $\cal F$ is a divisor $K_{\cal F}$ such that $\mathcal O_X(-K_{\cal F})\simeq  T_{\mathcal F}$.  
A {\bf rank one foliated  pair} $(\cal F, \Delta)$ is a pair of a foliation $\cal F$ of rank one and a $\bb Q$-divisor $\Delta\ge 0$
such that $K_{\cal F}+\Delta$ is $\bb Q$-Cartier. 
We refer to \cite[Sections 2.1, 2.2 and 2.3]{CS20} for the classical notions for foliations, such as their singularities and invariant subvarieties.

\begin{lemma}\label{l_cone}
Let $X$ be a normal projective variety and let $(\cal F,\Delta)$ be a rank one  foliated  pair on $X$ with log canonical singularities and such that $K_{\cal F}$ is $\mathbb Q$-Cartier. Assume that  $\Delta=A+B$ where $A$ is an ample $\bb Q$-divisor and $B\ge 0$ is a $\bb Q$-divisor. Assume that $K_{\cal F}+\Delta$ is not nef but there exists a $\bb Q$-divisor $H$ such that $K_{\cal F}+\Delta+H$ is nef. 
Let 
\[
\lambda \coloneq \inf\{t>0\mid K_{\cal F}+\Delta+tH\text{ is nef }\}.
\]

Then there exists a $(K_{\cal F}+\Delta)$-negative extremal ray $R=\mathbb R_+[C]$ such that 
$C$ is $\cal F$-invariant and 
$(K_{\cal F}+\Delta+\lambda H)\cdot C=0$. In particular, $\lambda\in \mathbb Q$. 
\end{lemma}
\begin{proof}
The proof is the same as the proof of  \cite[Lemma 9.2]{CS21}, as a consequence of the cone theorem for rank one foliations (cf. \cite[Theorem 1.2] {CS23b}).
\end{proof}

The following result will be used  in the proof of both Theorem \ref{t_main1} and Theorem \ref{t_main0}: 

\begin{proposition}
\label{p_semi_ample_fibr}
Let $X$ be a  normal projective variety 
and let $(\cal F,\Delta)$ be a rank one  foliated  pair on $X$ with log canonical singularities and such that 
$\cal F$ is algebraically integrable. 
 Assume that   $H\coloneq K_{\cal F}+\Delta$ is nef and that there exists an $\cal F$-invariant curve 
$\xi$  passing through a general point of $X$  such that $H\cdot \xi=0$. 

Then $H$ is semi-ample.
\end{proposition}
\begin{proof}
Since $(\cal F,\Delta)$ is log canonical, \cite[Lemma 2.5]{CS23b} implies that no component of $\Delta$ is $\cal F$-invariant. 
Let $p\colon \overline{X}\to X$ be a $(*)$ modification as in \cite[Theorem 3.10]{ACSS} so that, in particular,  $\overline{X}$ is klt and $\mathbb Q$-factorial and if $\overline{\cal F} \coloneqq p^{-1}\cal F$  then 
$\overline{\cal F}$  is induced by an equidimensional morphism $q\colon \overline{X}\to Z$ onto a smooth projective variety $Z$ of dimension $\dim X-1$. 
If $\overline \xi$ is the strict transform of $\xi$ in $\overline X$ then $\overline \xi$ is a fibre of $q$.
 In particular, 
% $\overline \xi$ is a moving curve which is $K_{\overline{X}}$-negative and 
 $p^*H$ is numerically trivial over $Z$. 
Moreover, 
if $\overline {\Delta}$ is the strict transform of $\Delta$ on $\overline X$, then we may  write 
\[
K_{\overline {\cal F}}+\overline\Delta +E = p^*(K_{\cal F}+\Delta)
\]
where $E$ is the  sum of all the $p$-exceptional divisors which are not $\overline{\cal F}$-invariant. 
In particular, $(\overline{\cal F},\overline\Delta+E)$ is log canonical. 

Let $G$ be the divisor associated to $(\overline{\cal F},\overline\Delta+E)$ (cf. \cite[Definition 3.5]{ACSS}) and let $\Gamma\coloneq \overline\Delta+E+G$. 
By \cite[Proposition 3.6]{ACSS}, we have that
\[
K_{\overline X}+\Gamma\sim_{f,\mathbb Q}
K_{\overline{\cal F}}+\overline\Delta+E \sim_{f,\mathbb Q} 0.
\]
Since $f$ is flat of relative dimension one, it follows that 
\[
K_{\overline {\cal F}}+\overline \Delta+E\sim_{\mathbb Q}q^*M_Z
\]
for some $\mathbb Q$-divisor $M_Z$ on  $Z$ and \cite[Proposition 3.6]{ACSS} implies that $M_Z$ is the moduli part of $q$ with respect to $(\overline X,\Gamma)$. 
By \cite[Theorem 4.3]{ACSS}, we have that $(\overline X/Z,\Gamma)$ is BP stable over $Z$ (cf. \cite[Definition 2.5]{ACSS}). 
Thus, \cite[Theorem 8.1]{PS09} implies that $M_{Z}$ is semi-ample and the result follows. 
\end{proof}

\section{Base point free theorem in dimension three}

The goal of this section is to prove Theorem \ref{t_main1}. 
The following two results are a slight generalisation of \cite[Lemma 5.7]{CS20} and \cite[Lemma 5.15]{CS20} respectively: 

\begin{lemma}\label{l_sing}
Let $X$ be a  normal variety and let $\cal F$ be a rank one foliation with canonical singularities. 
Let $C$ be an $\cal F$-invariant curve such that $K_{\cal F}\cdot C<0$. Assume that  $C$ does not move in a family of $\cal F$-invariant curves covering $X$.

Then there exists exactly one closed point $P\in C$ such that  $\mathcal F$ is not
terminal at $P$. Moreover, there exists at most one closed point $Q\in C\setminus \{P\}$ such that $K_{\cal F}$ is not Cartier at $Q$. 
\end{lemma}

\begin{proof}
By \cite[\S 4.1]{bm16} and since $K_{\cal F}\cdot C<0$ we have  that $C$ is not contained in $\sing \cal F$ and, therefore, $\cal F$ is terminal at a general closed point of $C$. 
By \cite[Proposition 3.3]{CS20}, there exists a closed point $P\in C$ such that $\cal F$ is not terminal at $P$. By definition of invariance with respect to $\cal F$, we have  that $K_{\cal F}$ is Cartier at a general point of $C$. Since $K_{\cal F}\cdot C<0$,  \cite[Proposition 2.13]{CS20} implies our claims. 
\end{proof}

\begin{lemma}\label{l_irreducible}
Let $X$ be a  projective threefold with $\mathbb Q$-factorial klt singularities and let $\cal F$ be a rank one foliation with canonical singularities. Let $C_1,C_2$ be  $\cal F$-invariant  curves on $X$ such that $C_1\cap C_2\neq \emptyset$ and such that $K_{\cal F}\cdot C_i<0$, for $i=1,2$. Assume that $C_1$ spans an extremal ray $R:=\mathbb R_+[C]$ of $\overline {\rm NE}(X)$ such that $\loc (R)$ is one dimensional and $C_2$ is not contained in $\loc (R)$. 

Then for a general point $x\in X$ there exist an $\mathcal F$-invariant curve $\xi_x$ in $X$ passing through $x$ and rational numbers $a,b\ge 0$ such that $[a C_1+b C_2]=[\xi_x]$ in $\overline{\rm NE}(X)$.
\end{lemma}
\begin{proof} 
 By Lemma \ref{l_sing}, we may assume that there exists exactly one closed point $P\in C_2$ such that  $\mathcal F$ is not
terminal at $P$. Moreover, there exists at most one closed point $Q\in C_2\setminus \{P\}$ such that $K_{\cal F}$ is not Cartier at $Q$. Note that, since $C_1\cap C_2$ is $\cal F$-invariant, it follows that $\cal F$ is terminal at every closed point of $C_2$ which is not contained in $C_1$. 

Let $\phi\colon X\dashrightarrow X'$ be the flip assiciated to $R$ and whose existence is guaranteed by 
 \cite[Theorem 8.8]{CS20}. Let $\cal F'\coloneq\phi_*\cal F$ and let $C'_2$  be the strict transform of $C_2$ in $X'$. By the negativity lemma (e.g. see \cite[Lemma 2.7]{CS20}) it follows that $\cal F'$ is terminal at any closed point of  $C'_2$ and that there are at most two closed points in $C'_2$ along which $K_{\cal F'}$ is not Cartier. Thus,  \cite[Proposition 2.13]{CS20} implies that $K_{\cal F'}\cdot C'_2<0$ and \cite[Proposition 3.3]{CS20} implies that $C'_2$ moves in a family of curves covering $X'$. Therefore, our claim follows. 
\end{proof}

\begin{theorem}\label{t_bpf}
Let $X$ be a  projective threefold with $\mathbb Q$-factorial klt singularities 
and let $(\cal F,\Delta)$ be a rank one  foliated  pair on $X$ with log canonical singularities. Assume that  $\Delta=A+B$ where $A$ is an ample $\bb Q$-divisor and $B\ge 0$ is a $\bb Q$-divisor. 
Suppose that $K_{\cal F}+\Delta$ is nef.  

Then $K_{\cal F}+\Delta$ is semi-ample.
\end{theorem}
\begin{proof} Let $H\coloneq K_{\cal F}+\Delta$. 
Assume first that $H$ is not big. In particular,  $K_{\cal F}$ is not pseudo-effective and  
Miyaoka's theorem (e.g. see \cite[Theorem 7.1]{brunella00}) implies that $\mathcal F$ is algebraically integrable. 
 By the bend and break (cf. \cite[Corollary 2.28] {Spicer20}) it follows that there exists a rational curve $\xi$ passing through a general point of $X$,  which is tangent to $\cal F$ and such that $H\cdot \xi=0$ (e.g. see the proof of \cite[Theorem 6.3]{Spicer20} for more details). 
By Proposition \ref{p_semi_ample_fibr} it follows that $H$ is semi-ample, as claimed.

\medskip 

Thus, we may assume that $H$ is big.  If $K_{\cal F}+\frac{1}{2}A+B$ is nef then $H$ is ample and there is nothing to prove.   So suppose that $K_{\cal F}+\frac{1}{2}A+B$ is not nef. By Lemma \ref{l_cone},  there exists a  $(K_{\cal F}+B)$-negative extremal ray $R=\mathbb R_+[C]$  such that $C$ is $\cal F$-invariant and  $(K_{\cal F}+\Delta)\cdot C=0$. Since $(\cal F,\Delta)$ is log canonical, it follows that $C$ is not contained in the support of $B$ and therefore $K_{\cal F}\cdot C<0$. Since $H$ is big, we have that  $\loc R\neq X$. 
By \cite[Corollary 8.5]{CS20}, we may  assume that $\cal F$ is canonical along $C$. 

Assume that $\loc R$ is a surface and let $\varphi\colon X \rightarrow X'$ be the birational contraction associated to $R$ and whose existence is guaranteed by \cite[Theorem 8.8]{CS20}. Note that $X'$ is $\mathbb Q$-factorial. 
Let $\Delta'\coloneqq \varphi_*\Delta=A'+B'$ where $A'\coloneqq \varphi_*A$   and $B'\coloneqq \varphi_*B\ge0$. 
Let $E$ be the exceptional divisor. We first show that $\varphi(E)$ is a closed point. Indeed, assume by contradiction that  $\xi\coloneq \varphi(E)$ is a curve and let $F$ be a valuation over $X'$ centred inside $\xi$. By the negativity lemma   (e.g. see \cite[Lemma 2.7]{CS20}) and \cite[Lemma 8.3]{CS20}, we have that $a(F,\mathcal F',B')\ge 0$. In particular, $(\mathcal F',B')$ is canonical along $\xi$ and \cite[Lemma 2.6]{CS20} implies that $F$ is invariant. Thus, by applying the negativity lemma again, we get that $a(F,\mathcal F',B')> 0$ and, in particular, $\cal F'$ is terminal along $\xi$. Thus, \cite[Lemma 2.9]{CS20} implies that $\cal F'$ is smooth along $\xi$ and, therefore,  there exits an $\cal F'$-invariant curve $T'$ passing through a general point $\eta$ of $\xi$. Note that $T'$ is distinct from $\xi$. Let $T$ be the strict transform of $T'$ in $X$. Then $\varphi^{-1}(\eta)\cap T\subset \sing \cal F$. In particular, since $A$ is ample, it follows that  $\sing \cal F$ contains a curve which intersects $A$, contradicting the fact that $(\cal F,A+B)$ is log canonical. Hence, we have shown that $\varphi(E)$ is a closed point and, in particular, it follows that $A'\coloneqq \varphi_*A$ is ample. 
Let $\cal F' \coloneqq \varphi_*\cal F$. Note that $(\cal F',\Delta')$ is log canonical and $K_{\cal F}+\Delta \sim_{\bb Q}  \varphi^*(K_{\cal F'}+\Delta')$ and so
$K_{\cal F}+\Delta$ is semi-ample provided $K_{\cal F'}+\Delta'$ is. 

Thus, by proceeding by induction on the Picard number of $X$, we may assume that for each $(K_{\cal F}+\Delta)$-trivial extremal ray $R$, we have that $\loc (R)$ is one dimensional. Let $R_1,\dots,R_q$ be all such extremal rays. 
 We claim that they are pairwise disjoint. Indeed, suppose not. Then, after possibly reordering, we may assume that there exist two curves $C_1$ and $C_2$ such that $[C_i]\in R_i$ for $i=1,2$ and $C_1\cap C_2\neq \emptyset$.  
By \cite[Remark 2.31]{CS20}, we have that $C_1$ and $C_2$ are $\cal F$-invariant. 
Since $(\cal F,\Delta)$ is log canonical, as in the proof of \cite[Lemma 4.7]{CS23b}, it follows that $C_1$ and $C_2$ are not contained in the support of $\Delta$ and, therefore, $K_{\cal F}\cdot C_i<0$. 
Thus, 
by Lemma \ref{l_irreducible}, it follows that there exist an $\mathcal F$-invariant curve $\xi_x$ in $X$ passing through a general point $x\in X$ and rational numbers $a,b\ge 0$ such that $[a C_1+b C_2]=[\xi_x]$ in $\overline{\rm NE}(X)$. In particular, $H\cdot \xi_x=(K_{\cal F}+\Delta)\cdot \xi_x=0$ and, therefore,  $H$ is not big, a contradiction. Thus, our claim follows. 

\medskip

Let $ {\rm Null}(H)$ be the exceptional locus of $H$ (e.g. see \cite[Section 2.12]{CS20}). Since $H$ is big, it follows that $ {\rm Null}(H)\neq X$. Assume by contradiction that  ${\rm Null}(H)$ contains a surface $S$. Let $\nu\colon S^{\nu}\to S$ be the normalisation of $S$.  By \cite[Corollary 4.9]{CS23b}, $S$ is $\cal F$-invariant and by 
\cite[Proposition-Definition 3.7]{CS23b}, we may write
\[
\nu^*(K_\cal F+\Delta)=K_{\mathcal F_{S^{\nu}}}+\Delta_S
\] 
where $\mathcal F_{S^{\nu}}$ is the restricted foliation on $S^{\nu}$ and  $\Delta_S\ge 0$ is a $\mathbb Q$-divisor on $S^{\nu}$. By the bend and break (cf. \cite[Corollary 2.28] {Spicer20}) it follows that there exists a rational curve $\xi$ passing through a general point of $S^{\nu}$, which is tangent to $\cal F_{S^{\nu}}$ and such that $(K_{\mathcal F_{S^{\nu}}}+\Delta_S)\cdot \xi=0$. In particular, if $\xi'=\nu(\xi)$ then $(K_{\cal F}+\Delta)\cdot \xi' =0$ and 
 $S\cdot \xi'<0$.  Thus, there exists an extremal ray $R$ of $\overline{{\rm NE}}(X)$, which is $(K_{\cal F}+\Delta)$-trivial, and such that $\loc (R)\subset S$.
Moreover, it follows that $\cal F_{S^\nu}$ is algebraically integrable and, therefore,  there exists a birational morphism of normal surfaces $h\colon T\to S^\nu$  such that the foliation $h^{-1}\cal F_{S^{\nu}}$ is induced by a morphism $g\colon T\to B$ where $B$ is a smooth curve. In particular, if $F$ is the general fibre of $f$ then $h_*F$ is numerically equivalent to $\xi$.  
By \cite[Remark 2.31]{CS20}, it follows that if $C_1$  is a curve contained in  $\loc (R)$ then $C_1$ is not contained in the singular locus of $\cal F$ and $C_1$ is $\cal F$-invariant. Thus, if $C'_1$ is the strict transform of $\nu^{-1}(C_1)$ in $T$, then $C'_1$ is contained in a fibre of $f$. On the other hand, 
since $\dim \loc (R)=1$, it follows that, for any fibre $F_0$ of $f$, we have that  $[h_*F_0]$ is not contained in $R$. Thus, $C_1$ intersects a curve $C_2$ which is $\cal F$-invariant and such that $C_2$ is not contained in $\loc(R)$ and $(K_{\cal F}+\Delta)\cdot C_2=0$. As above, since $(\cal F,\Delta)$ is log canonical, it follows that $K_{\cal F}\cdot C_i<0$ for $i=1,2$.  
 Thus,  Lemma \ref{l_irreducible} yields a contradiction and ${\rm Null}(H)$ does not contain any surface.

 Let 
\[
\Sigma \coloneq  \bigcup_{i=1}^q \loc (R_i).
\]
 By \cite[Theorem 8.8]{CS20}, it follows that each component of $\Sigma$ can be contracted and so by Artin's theorem  \cite[Theorem 6.2]{Artin70} there exists a morphism $\psi\colon X\to Y$ in the category of algebraic spaces such that 
$\exc \psi = \Sigma$ 
and $H=\psi^*H_Y$ for some $\mathbb Q$-Cartier $\mathbb Q$-divisor $H_Y$ on $Y$. If $H_Y$ is ample, then $H$ is semi-ample and we are done. 

Assume now, by contradiction, that $H_Y$ is not ample.  Since ${\rm Null}(H)$ does not contain any surface, it follows that $H_Y|_T$ is big for any surface $T$ on $Y$ and, by Nakai-Moishezon theorem, there exists an extremal ray $R_Y$ of $\overline{{\rm NE}(Y)}$ which is $H_Y$-trivial. Thus, there exists a $H$-trivial extremal ray $R_X$ such that $[\psi_*\xi] \in R_Y$ for all $\xi\in R_X$. By construction, we have that $\loc (R_X)\subset \Sigma = \exc \psi$, a contradiction. 
\end{proof}

\section{Numerically trivial log canonical foliated  pairs }

The goal of this section is to prove the following: 

\begin{theorem}\label{t_main}
Let $X$ be a normal projective $\mathbb Q$-factorial klt variety, let $\mathcal F$ be a rank one foliation on $X$ and let $\Delta \ge 0$ be a $\mathbb Q$-divisor such that $(\mathcal F, \Delta)$ is log canonical and 
$K_{\mathcal F}+\Delta \equiv 0$.  

Then $K_{\mathcal F}+\Delta \sim_{\mathbb Q}0$.
\end{theorem}

\begin{proof}
We prove the theorem using a case by case analysis.

{\bf Case 1:  $\Delta \neq 0$ or $\mathcal F$ does not have canonical singularities.} We first show that, in both cases, 
$\cal F$ is algebraically integrable. 
 If $\Delta\neq 0$, then  $K_{\mathcal F}$ is not pseudo-effective and 
Miyaoka's theorem (e.g. see \cite[Theorem 7.1]{brunella00}) implies that $\mathcal F$ is algebraically integrable.
Assume now that $\Delta=0$ and 
$\cal F$ does not have canonical singularities. 
Then \cite[Corollary 3.8]{lpt18} implies that $\cal F$ is uniruled  (note that while {\it loc. cit.} is stated for smooth varieties, the proof applies equally well in our setting) and our claim follows also in this case. 
We may then conclude by Proposition \ref{p_semi_ample_fibr}.

\medskip

{\bf Case 2: $\Delta = 0$ and $\mathcal F$ has canonical singularities.}  
let ${\rm Alb}\colon X \to  A$ be the Albanese morphism
(e.g. see \cite[Lemma 8.1]{Kawamata85a}) and let 
\[
{\rm Alb}\colon X \xrightarrow{a} Z \to A
\] be its Stein factorisation.  
Since ${\rm Pic}^0(X) = {\rm Pic}^0(A)$, if $m>0$ is a sufficiently positive integer then $mK_{\mathcal F}$ is Cartier and there exists a line bundle
$L$ on $Z$ such that $\cal O_X(mK_{\mathcal F}) = a^*L$.

Either $\mathcal F$ is generically transverse to the fibres of $a$, or $\mathcal F$ is tangent to the fibres of $a$ (equivalently, $T_{\mathcal F} \subset T_{X/Z}$).

\medskip

{\bf Case 2.a: $\mathcal F$ is generically transverse to the fibres of $a$.} In this case, the composition
\[{\rm Alb}^*\Omega^1_A \to a^*\Omega^1_Z\to  \Omega^1_X \to \mathcal O_X(K_{\mathcal F})\]
is non-zero. Since ${\rm Alb}^*\Omega^1_A \cong \mathcal O_X^{\dim A}$ we see that $H^0(X, \mathcal O(K_{\mathcal F})) \neq 0$ and we may conclude.

\medskip

{\bf Case 2.b: $\mathcal F$ is tangent to the fibres of $a$, i.e., $T_{\mathcal F} \subset T_{X/Z}$.}
We denote by $X_z \coloneq a^{-1}(z)$ the fibre of $a$ at $z\in Z$ and, for a general  $z\in Z$, we denote by $\cal F_z$ the restricted foliation on $X_z$ (cf. \cite[Proposition-Definition 3.12]{CS23b}). 

 % By induction on the dimension, 
 %\footnote{P: do we  need to use induction? I believe that since $X$ is klt (and therefore with rational singularities) we have that 
 %$Pic^0(X)=Pic^0(A)$ and therefore some multiple of $K_{\cal F}$ is the pull back of a line bundle $L$ on $A$. I actually suggest that we don't even distinguish case 2.a from case 2.b. We can just from the very beginning that we may assume that $mK_{\cal F}=a^*L$ and, in particular, we may assume that  $A$ is not a point }
 % we may assume that
%for a general $z\in Z$, here exists an integer $m>0$ so that $mK_{\mathcal F}$ is Cartier and  $0 \sim mK_{\mathcal F_z} \sim  mK_{\mathcal F}|_{X_z}$.
%By upper semi-continuity of the fibre dimension of $a_*\mathcal O_X(mK_{\mathcal F})$ it follows that, for all $z \in Z$, we have  $h^0(X_z, \mathcal O_{X_z}(mK_{\mathcal F})) \ge 1$.  Since $mK_{\mathcal F}|_{X_z} \equiv 0$ and $h^0(X_z, \mathcal O_{X_z}(mK_{\mathcal F}) \ge 1$ for all $z\in Z$,  we conclude that in fact $mK_{\mathcal F}|_{X_z} \sim 0$ for all $z \in Z$.
%Thus,  $\mathcal O_X(mK_{\mathcal F})\simeq a^*L$ for some $L \in {\rm Pic}^0(A)$.  

Choose $M\in  {\rm Pic}^0(A)$ such that $M^{\otimes m} = L$.
We form the relative index one cover associated to $K_{\mathcal F}$ as follows.
Consider the sheaf 
\[\mathcal A \coloneq \bigoplus_{i = 0}^{m-1} \mathcal O_X(-iK_{\mathcal F}) [\otimes] a^*M^{\otimes i}\]
where $[\otimes]$ denotes the reflexive tensor product.
Using the isomorphism $\mathcal O_X(-mK_{\mathcal F})\otimes M^{\otimes m} \to \mathcal O_X$ we equip $\mathcal A$ with the structure of an
$\mathcal O_X$-algebra.  Let $X' \coloneq {\rm Spec}_X\mathcal A$ and let $r\colon X' \to X$ be the natural morphism.

Note that $r\colon X' \to X$ is quasi-\'etale when restricted to the generic fibre, and in particular, the ramification of $r\colon X' \to X$ is supported on $\mathcal F$-invariant divisors.  
By \cite[Lemma 3.4]{Druel19}, cf. \cite[Proposition 2.20]{CS20}, we see that $K_{r^{-1}\mathcal F} = r^*K_{\mathcal F}$ and  therefore it suffices to prove that $K_{r^{-1}\mathcal F} \sim_{\mathbb Q} 0$. By \cite[Lemma 2.8]{CS20} we have that $r^{-1}\cal F$ has canonical singularities.
Thus, up to replacing $(X, \mathcal F)$ by $(X', r^{-1}\mathcal F)$ we may freely assume that $\mathcal O_X(K_{\mathcal F}) \simeq a^*L$ where $L$ is a line bundle, and in particular, $K_{\mathcal F_z} \sim 0$ for general $z \in Z$, where $\mathcal F_z$ is the restricted foliation on $X_z :=a^{-1}(z)$.  Since $K_{\mathcal F_z} \sim 0$, we have that $\cal F_z$ is generated by a global vector field, which we will denote $\delta_z$.

Let $\mu\colon \widehat{X} \to X$ be a functorial resolution of singularities (cf.   \cite[Notation 4.5]{MR2581247}). 
From \cite[Corollary 4.7]{MR2581247} we deduce that $K_{\mu^{-1}\mathcal F} = \mu^*K_{\mathcal F}$, so up to replacing $X$
by $\widehat{X}$ we may freely assume that $X$ is smooth.  

\medskip 

{\bf Case 2.b.i: A component of $\sing \mathcal F$ dominates $Z$.}
Let $S$ be a component of $\sing \mathcal F$ which dominates $Z$.  By \cite[\S 4.1]{bm16}
we see that $K_{\mathcal F}|_S$ is semi-ample.  Since $\mathcal O_S(K_{\mathcal F}) \simeq (a|_S)^*L$ we deduce that 
$L$ is torsion, and we may conclude.

\medskip 

{\bf Case 2.b.ii: $\sing \mathcal F_z = \emptyset$ for a general point $z \in Z$.}
In this case, by \cite[Remark 1.5 and Theorem 3.2]{MR2928843} up to an \'etale cover, either $X_z$ is a suspension over an abelian variety, or
%\footnote{P: stupid remark, but since we already wrote "up to an etale cover", can we just say that it is a product? }
 $X_z \cong T_z\times F_z \to X_z$ where $F_z$ is an abelian variety and $T_z$ admits no global vector fields.
In either case, (up to an \'etale cover) there is a morphism $p\colon X_z \to F_z$ where $F_z$ is an abelian variety and the pushforward of $\delta_z$ is a global vector field on $F_z$.

Thus, after replacing $X$ by a finite cover which is ramified only on fibres of $X \to Z$, we may assume that we have a morphism $f\colon X \to F$ over $Z$ such that a general fibre of $b\colon F \to Z$ is an abelian variety
and there exists a rank one foliation $\mathcal G$ on $F$ such that $\mathcal G_z$ is defined by a global vector field for general $z \in Z$. In particular, we have a non-trivial natural map $T_{\cal F}\to f^*T_{\cal G}$ and therefore there exists a divisor $B\ge 0$ such that $K_{\cal F}=f^*K_{\cal G}+B$. Since $K_{\cal G}$ is pseudo-effective, it follows that $K_{\mathcal F} \sim f^*K_{\mathcal G}$.  
%\footnote{P:  why do we have  $K_{\mathcal F} \sim f^*K_{\mathcal G}$? could we have some contribution over the non general point which is numerically trivial?  C: What I had in mind is the following, but I'll think if there is a slicker way to say things: We have an equality $K_{\mathcal F} \sim f^*K_{\mathcal G}+B$ where $B \ge 0$.  We know that $\mathcal G$ can't be uniruled and so $K_{\mathcal G}$ is psef.  Since $K_{\mathcal F} \equiv 0$ we get $B = 0$.}
Thus, we may freely replace $X$ by $F$ and  we may assume that a general fibre of $X \to Z$ is an abelian variety.

Next, note that if $C \subset Z$ is a general complete intersection curve, then the natural map ${\rm Pic}^0(Z) \to {\rm Pic}^0(C)$ is injective, and so to show that $L$ is torsion it suffices to show that $L\vert_C$ is torsion.
Thus we may freely replace $Z$ by $C$ and $X$ by $a^{-1}(C)$ and so may assume that $Z$ is a curve.
%Moreover,  we claim that we may freely replace $Z$ by a general complete intersection curve $C \subset Z$, and $X$ by a semi-stable reduction of $X\times_CZ$.
% be a semi-stable reduction in codimension one, cf. \cite[Theorem 4.3]{ambro04}.
Next, let us consider a semi-stable reduction of $X \to Z$ which is guaranteed to exist by \cite{kkms73}, i.e., a diagram

\begin{center}
\begin{tikzcd}
X' \arrow[r, "\alpha"] \arrow[d] & X \arrow[d] \\
Z' \arrow[r, "\sigma"] & Z

\end{tikzcd}
\end{center}
where $\sigma\colon Z' \to Z$ is finite and $\alpha\colon X' \to X$ is the composition of a resolution of singularities of $X\times_ZZ'$ together with the natural projection $X\times_ZZ' \to X$.
By taking our resolution of singularities to be functorial and noting that the ramification of $X' \to X$ 
is $\alpha^{-1}\mathcal F$-invariant, arguing as above, we again see that $K_{\alpha^{-1}\mathcal F} = \alpha^*K_{\mathcal F}$.  Thus we may freely replace $X$ by $X'$ and so may assume that $X \to Z$ is semi-stable.

If we push forward the morphism $\Omega^1_{X/Z} \to \mathcal O_X(K_{\mathcal F})$ along $a$
we get a generically surjective morphism $a_*\Omega^1_{X/Z} \to L$.
If we let $U \subset Z$ be an open subset such that $X_U:=a^{-1}(U) \to U$ is a smooth family of abelian varieties, and note that we have splitting $a_*\Omega^1_{X/Z}|_U \simeq L|_U \oplus M$ for some vector bundle $M$ on $U$. 
Let $D \coloneq X \setminus a^{-1}(U)$ and note that $(X, D)$ is an snc pair.
Since every component of $D$ is vertical with respect to $a$, each component of $D$ is $\cal F$-invariant and therefore we  have a morphism $\Omega^1_{X/Z}(\log D) \to \cal O_X(K_{\mathcal F})$.

%We first show that $\Omega^1_{X/Z}(\log D)$ is a semi-positive vector bundle.  
Let $H$ be the Deligne canonical extension of 
$R^1a_*\mathbb C_{X_U}$.
By \cite[Corollary, pg. 130]{MR756849} there is a decreasing filtration on $H$, extending the Hodge filtration on $R^1a_*\mathbb C_{X_U}$, such that the bottom piece of this filtration is $a_*\Omega^1_{X/Z}(\log D)$.
%By \cite[\S 2.10]{kollar86b}, cf. also \cite[Proposition 2.9]{kollar86b}, (note that using notation in \cite{kollar86b}) we have $a_i = 1$)  there is a decreasing filtration on $H$, extending the Hodge filtration on $R^1a_*\mathbb C_{X_U}$, such that the bottom piece of this filtration is $a_*\Omega^1_{X/Z}(\log D)$.
%\footnote{C: if I understand his argument this is ok.  maybe it only shows that the canonical extension is a subsheaf of $a_*\Omega^1_{X/Z}(\log D)$, but this would be fine for what we need later.}
%It follows $a_*\Omega^1_{X/Z}(\log D)$ is a semi-positive vector bundle, see for instance \cite[Theorem 3]{MR3167580}.

We next note that 
the pushforward of $\Omega^1_{X/Z}(\log D) \to \mathcal O_X(K_{\mathcal F})$ gives a generically surjective morphism $a_*\Omega^1_{X/Z}(\log D) \to L$.  
%\footnote{In fact, this morphism is surjective: if, supposing for sake of contradiction, it were not surjective, then $\deg L>0$, contrary to hypothesis. C: I don't think we actually use this observation}

We will now show that $L$ is a local system.
As in \cite[Proof of Lemma]{MR510945} the natural Hermitian metric on $R^1a_*\mathbb C_{X_U}$ canonically determines Hermitian metrics on  $a_*\Omega^1_{X_U/U}$ and $L|_U$, such that curvature form with respect to these metrics is semi-positive.  Denote by $h_L$ the Hermitian metric on $L|_U$ and denote by $\Theta$ the corresponding curvature form 
on $L|_U$.
We then have \[\deg L = \int_{U} \Theta+\sum_{P \in Z \setminus U} a_P\]
where $a_P$ is the local exponent of $L$ at $P$ (see for instance \cite[Lemma 21]{kawamata81}).
We note that $a_P \ge  0$.  Indeed, as observed in \cite[Paragraph before Lemma 21]{kawamata81}
$a_P$ is determined by the following estimate $h_L(s_P, s_P) = O(|t|^{-2a_P}(\log |t|)^{-2b_P})$
where $t$ is a local coordinate on a neighbourhood of $P$ and $s_P$ is generator of $L$ in a neighbourhood of $P$.  Recall (from the local description of the canonical extension) that a section of
 $R^1a_*\mathbb C_{X_U}$ (resp. $a_*\Omega^1_{X_U/U}$) extends to $H$ (resp. $a_*\Omega^1_{X/Z}(\log D)$) provided it has logarithmic growth near $P$.  It follows that 
the local section of $s_P$ has at least logarithmic growth near $P$.

Since $\Theta$ is semi-positive and $\deg L = 0$ we deduce that in fact $\Theta = 0$ and $a_P = 0$.
Since $\Theta = 0$, 
\[L|_{U} \subset a_*\Omega^1_{X_U/U} \subset R^1a_*\mathbb C_{X_U}\] is a local subsystem of $R^1a_*\mathbb C_{X_U}$.
%\footnote{C: note to self: this is because $\Theta$ restricts to the semi-postiive curvature form which can be defined on the bottom piece of the hodge filtration.  If a bundle has zero curvature with respect to this form, then (using the calculations of Griffiths) we see that 
%the sections of the bundle are flat with respect to the Gauss-Manin connection.  I suppose what we need to know to make the argument above work is that the semi-positive curvature form on $a_*\Omega^1_{X/Z}(\log D)$ restricts to the semi-positive curvature form 
%defined on the bottom piece of the Hodge filtration}
Since $a_P = 0$ the local monodromy of $L_U$ around $P$ is trivial.  This implies that in fact $L$ is a local system. 
By \cite[Corollaire 4.2.8.iii.b]{MR498552}, some power $L|_U^{\otimes m}$ is the trivial local system.   Since $a_P = 0$, the monodromy around $P$ is trivial and so it follows that in fact $L^{\otimes m}\simeq 0$, as required.
\end{proof}

Theorem \ref{t_main} has the following interesting Corollary.  We thank F. Bernasconi for pointing this out to us.

\begin{corollary}
\label{cor_hyperbolicity}
Let $X$ be a normal projective $\mathbb Q$-factorial klt variety, let $\mathcal F$ be a rank one foliation on $X$ such that $\mathcal F$ is log canonical and  
$K_{\mathcal F}\equiv 0$.

Then, for a general point $x \in X$ there exists a holomorphic map $f\colon \mathbb C \to X$ such that 
$x \in f(\mathbb C)$ and the image of $f$ is tangent to $\mathcal F$.
\end{corollary}
\begin{proof}
By Theorem \ref{t_main} $K_{\mathcal F}\sim_{\mathbb Q} 0$.  So, up to replacing $X$ by the index one cover associated to $K_{\mathcal F}$, we may assume that $K_{\mathcal F} \sim 0$.
Moreover, up to replacing $X$ by a functorial resolution of singularities, we may assume that $X$ is smooth.

Since $K_{\mathcal F} \sim 0$ we see that  $\mathcal F$ is generated by a global vector field $v \in H^0(X, T_X)$.  Since $H^0(X, T_X)$ is the Lie algebra of $\Aut (X)$, 
for a general point $x \in X$, we have the exponential map $\exp_x\colon H^0(X, T_X) \to X$ such
that $\exp_x(0) = x$.  We take $f\colon \mathbb C \to X$ to be the restriction of $\exp_x$ to the subspace $\mathbb Cv \subset H^0(X, T_X)$. 
\end{proof}

\bibliography{math.bib}
\bibliographystyle{alpha}
\end{document}